\documentclass[11pt,a4paper,reqno]{amsart}
\usepackage{amssymb,enumerate,a4wide}
\usepackage{booktabs}
\usepackage[all]{xy}

\usepackage{hyperref}
\hypersetup{ 
  pdfauthor   = {Reynald Lercier, Christophe Ritzenthaler and Jeroen Sijsling}
  pdftitle    = {Fast computation of isomorphisms of hyperelliptic curves
       and explicit Galois descent}
  pdfsubject  = {Primary 13A50; Secondary 14Q05, 14H10, 14H25}
  pdfkeywords = {invariants ; covariants ; hyperelliptic curves ; binary form ; Galois descent ; isomorphism ;
moduli ; algorithm},
  backref=true, pagebackref=true, hyperindex=true, colorlinks=true,
  breaklinks=true, urlcolor=blue, linkcolor=blue, citecolor=blue, bookmarks=true,
  bookmarksopen=true}

\newcommand{\bedradyuk}{Bedratyuk2007}
\newcommand{\brouwerone}{BrouwerPopoviciu2010a}
\newcommand{\brouwertwo}{BrouwerPopoviciu2010b}
\newcommand{\CAQU}{CardonaQuer2005}
\newcommand{\CremonaFisher}{CremonaFisher2009}
\newcommand{\cronih}{Croni2002}
\newcommand{\dixmierlazard}{DixmierLazard1985/1986}
\newcommand{\dixmier}{Dixmier1990}
\newcommand{\earle}{Earle1971}
\newcommand{\fuertesgonzalez}{FuertesGonzalez2006}
\newcommand{\gordan}{Gordan1868}
\newcommand{\graceyoung}{GraceYoung1903}
\newcommand{\Hess}{Hess2004}
\newcommand{\huggins}{Huggins2007}
\newcommand{\hugginsphd}{Huggins2005}
\newcommand{\igusa}{Igusa1960}
\newcommand{\KedlayaUmans}{KedlayaUmans2008}
\newcommand{\koizumi}{Koizumi1972}
\newcommand{\LReight}{LercierRitzenthaler2008}
\newcommand{\LReleven}{LercierRitzenthaler2012}
\newcommand{\mestre}{Mestre1991}
\newcommand{\mumfordfogarty}{MumfordFogarty1982}
\newcommand{\olver}{Olver1999}
\newcommand{\serre}{Serre1968}
\newcommand{\shimura}{Shimura1972}
\newcommand{\shioda}{Shioda1967}
\newcommand{\VanRijnswou}{Rijnswou2001}
\newcommand{\vongall}{Gall1888}
\newcommand{\weil}{Weil1956}
\newcommand{\magma}{MAGMA}

\renewcommand{\tilde}{\widetilde}

\theoremstyle{plain}
\newtheorem{theorem}{Theorem}[section]

\newtheorem{proposition}[theorem]{Proposition}
\newtheorem{lemma}[theorem]{Lemma}

\theoremstyle{definition}
\newtheorem{definition}[theorem]{Definition}
\newtheorem{remark}[theorem]{Remark}
\newtheorem{example}[theorem]{Example}
\newtheorem{algorithm}[theorem]{Algorithm}

\newlength{\algindent}
\newlength{\alglabel}
\newcounter{stepcount}

\newenvironment{alglist}
{\quad\begin{list}{\arabic{stepcount}.}{\leftmargin=\algindent\labelwidth=\algindent\itemsep=\smallskipamount\usecounter{stepcount}}}
{\end{list}}

\newcommand{\algin}{\item[\emph{Input}:]}
\newcommand{\algout}{\item[\emph{Output}:]}

\newcounter{substepcount}
\newenvironment{algsublist}
{\quad\begin{list}{(\/\rlap{\alph{substepcount}}\phantom{d}\/)}{\usecounter{substepcount}}}
{\end{list}}

\newcounter{subsubstepcount}
\newenvironment{algsubsublist}
{\quad\begin{list}{\roman{subsubstepcount}.}{\usecounter{subsubstepcount}}}
{\end{list}}

\DeclareMathOperator{\Aut}{Aut}
\DeclareMathOperator{\Autred}{\overline{Aut}}
\DeclareMathOperator{\Gal}{Gal}
\DeclareMathOperator{\GL}{GL}
\DeclareMathOperator{\Isom}{Isom}
\DeclareMathOperator{\PGL}{PGL}
\DeclareMathOperator{\SL}{SL}
\DeclareMathOperator{\Stab}{Stab}
\DeclareMathOperator{\Sym}{Sym}

\newcommand{\Id}{\textup{id}}

\newcommand{\Qm}{Q}
\newcommand{\Xm}{X}

\newcommand{\FF}{\mathbb{F}}
\newcommand{\CC}{\mathbb{C}}
\newcommand{\NN}{\mathbb{N}}
\newcommand{\PP}{\mathbb{P}}
\newcommand{\QQ}{\mathbb{Q}}
\newcommand{\ZZ}{\mathbb{Z}}

\newcommand{\Qbar}{\bar{\QQ}}
\newcommand{\Fbar}{\bar{F}}

\newcommand{\AG}{\mathbf{A}}
\newcommand{\CG}{\mathbf{C}}
\newcommand{\DG}{\mathbf{D}}
\newcommand{\MG}{\mathbf{M}}
\newcommand{\SG}{\mathbf{S}}
\newcommand{\UG}{\mathbf{U}}
\newcommand{\VG}{\mathbf{V}}

\newcommand{\cc}{\mathcal{C}}
\newcommand{\ii}{\mathcal{I}}
\newcommand{\xx}{\mathcal{X}}

\newcommand{\frakI}{\mathfrak{I}}
\newcommand{\frakf}{\mathfrak{f}}
\newcommand{\frakc}{\mathfrak{c}}
\newcommand{\frakq}{\mathfrak{q}}

\newcommand{\Bord}{B_{\mathrm{order}}}
\newcommand{\Bdeg}{B_{\mathrm{degree}}}
\newcommand{\Bsing}{B_{\mathrm{singular}}}

\newcommand{\mybar}[1]{
  \mathchoice
  {#1\llap{$\overline{\phantom{\displaystyle\rm#1}}$}}
  {#1\llap{$\overline{\phantom{\textstyle\rm#1}}$}}
  {#1\llap{$\overline{\phantom{\scriptstyle\rm#1}}$}}
  {#1\llap{$\overline{\phantom{\scriptscriptstyle\rm#1}}$}}
}  
\renewcommand{\bar}{\mybar}

\title[Isomorphisms and field descent of hyperelliptic curves]
      {Fast computation of isomorphisms of hyperelliptic curves
       and explicit Galois descent}

\author{Reynald Lercier}
\address{DGA MI, La Roche Marguerite, 35174 Bruz, France;
         IRMAR, Universit\'e de Rennes~1, Campus de Beaulieu, 35042 Rennes, France} 
\email{reynald.lercier@m4x.org}

\author{Christophe Ritzenthaler}
\address{Institut de Math{\'e}matiques de Luminy, UMR 6206 du CNRS, Luminy, Case 907, 13288 Marseille, France}
\email{ritzenth@iml.univ-mrs.fr}

\author{Jeroen Sijsling}
\address{IRMAR, Universit\'e de Rennes~1, Campus de Beaulieu, 35042 Rennes, France} 
\email{sijsling@gmail.com}


\begin{document}

\begin{abstract}
We show how to speed up the computation of isomorphisms of hyperelliptic
curves by using covariants. We also obtain new theoretical and practical
results concerning models of these curves over their field of moduli.
\end{abstract}

\maketitle

\section{Introduction}

Let $\Xm_1$ and $\Xm_2$ be two curves of genus $g \geq 2$ over a
field~$k$. We wish to quickly determine the (possibly empty) set 
of isomorphisms between them.  The standard strategy mainly consists
of interpolating the isomorphisms at Weierstrass or small degree places,
depending on whether the characteristic of the field is zero or
positive~\cite{\Hess}. This yields algorithms of complexity at least
$O(g^6)$ in general, and at least $O(g^2)$ even in very favorable cases.

In this article we restrict to hyperelliptic curves with equations 
$\Xm_i\colon y^2=f_i(x)$ over a field $k$ of characteristic
different from~$2$. The issue can then be rephrased in terms of 
isomorphisms of degree $2g+2$ polynomials under the M\"obius action
of $\GL_2(k)$ (see Section \ref{sec:iso-general-fac}).
Our first contribution is to show how to compute the set of isomorphisms
in a much faster way by combining two new ideas. The first one uses the 
factorization of the M\"obius action into a diagonal matrix times a
second matrix whose diagonal coefficients are equal to~$1$. 
This idea allows us to perform the computation of the isomorphisms 
with only univariate polynomial calculations (see 
Section~\ref{sec:classical-approach}). The second idea
relies on a classical generalization of invariants, called 
\emph{covariants} (see Section~\ref{sec:other-geom-covar}). 
Using covariants, we can reduce our search for an isomorphism between 
$f_1$ and $f_2$ to the search for an isomorphism between polynomials of 
lower degree. This gives us an algorithm for generic hyperelliptic curves 
whose complexity is quasi-linear in $g$ (see 
Section~\ref{sec:gener-exampl-genus}). 
In the genus-$2$ and genus-$3$ cases, we analyze the
small locus of curves where our strategy fails (see Section
\ref{sec:all-hyper-genus}). The use of covariants was inspired by
work of van Rijnswou~\cite{\VanRijnswou}, who used
covariants, along with a miraculous isomorphism from representation theory,
to generically reduce the isomorphism question for ternary quartics
to that for binary quartics.

In a related direction, thanks to covariants, we get both theoretical and
practical results on Galois descent of hyperelliptic curves over their
field of moduli. As the terminology suggests, this issue is related to 
moduli spaces, namely as follows.

The use of invariants allows the construction of the coarse
moduli space of smooth curves admitting a suitable representation
(for example, hyperelliptic or planar) as a geometric quotient in 
the sense of Mumford~\cite{\mumfordfogarty}.  Such quotients have
been calculated explicitly; for instance, for genus-$2$ and 
genus-$3$ hyperelliptic curves, see~\cite{\igusa,\shioda}. 
Given a field~$k$, the $k$-points of these quotients correspond
to curves whose field of moduli, in the sense of
Definition~\ref{def:moduli}, is equal to~$k$ (up to a possible 
purely inseparable extension). This statement is probably well-known,
but we could not find it in the literature; therefore, we give
the link between these two definitions in Section~\ref{sec:desc-algebr-curv}.

A natural question is to determine when a curve descends to its
field of moduli, that is, when its field of moduli is also a field
of definition (and hence the smallest possible field of definition,
under inclusion).  Examples of curves that do not so descend 
were  constructed by Shimura~\cite{\shimura} and Earle~\cite{\earle},
among others.  However, curves of genus at most $1$ always descend
to their field of moduli, and models over the field of moduli can
be explicitly constructed.  Moreover, in the genus-$2$ case, although 
an obstruction to the descent may exist, as is shown in~\cite{\mestre}
and~\cite{\CAQU}, the question of explicit descent to the field of
moduli is solved. One of our aims is to obtain similar results in the
general hyperelliptic case.

Many theoretical results for the general case can be found
in~\cite{\huggins}.  In practice, though, computing an explicit model
of a given curve over its field of moduli can be a very hard task, as we
explain in Section~\ref{sec:classical-approach-2}. 
Indeed, for a given \emph{finite} Galois extension,  
Weil's criterion in~\cite{\weil} often leads to a computational answer;
the main difficulty in our context is to work out the
finite Galois extension over which a descent isomorphism is defined.  
As far as we know, there is no easy general way to find this
extension, except when $k$ is finite or when the geometric 
automorphism group of the curve is trivial.  Moreover, for
hyperelliptic curves there is a refinement of the descent
question --- namely, to ask for a descent to a model of the 
form $y^2=f(x)$ --- and this introduces additional difficulties. 

The `magic' of the covariant method is to reduce the descent problem
to lower genus, where a solution may be easier to determine 
(Theorem \ref{th:covdescent}). In the genus-$1$ case, for example,
there is always an explicit model over the field of moduli and we can 
quickly determine a descent isomorphism to this model, thanks to the
first part of our work. It turns out that in suitable cases, this descent
induces a descent of the original hyperelliptic curve to its field of moduli.

We illustrate this descent to the field of moduli for genus-$3$
hyperelliptic curves with automorphism group $(\ZZ/2\ZZ)^3$,
a case which remained unsolved in~\cite{\LReleven};
see Section~\ref{sec:c23-case}. 
We also look at the case of genus-$3$ hyperelliptic curves
with automorphism group $(\ZZ/2\ZZ)^2$;
in this case the field of moduli is not always a field of definition,
and we prove that we can always find a model over an at most quadratic
extension of the field of moduli.  Finally,
in Section~\ref{sec:fuertes-examples} we show that our method
can be used to descend families of curves with the example of a
$3$-dimensional family of genus-$5$
hyperelliptic curves from \cite{\fuertesgonzalez}.

We stress that we are merely beginning to exploit the full strength of these
new ideas. An article on nonhyperelliptic curves is in progress. We are also
developing a general version of van Rijnswou's algorithms that is much more
effective over finite fields and number fields. Finally, we seek to obtain new
theoretical and practical descent results by analyzing the influence of twists
on covariants.

We have implemented our algorithms in Magma~\cite{\magma}; the resulting
programs, together with other useful scripts and some output that was too
large to include in this paper, may be found online at

\centerline{\url{http://perso.univ-rennes1.fr/christophe.ritzenthaler/programme/hyp-desc.tgz}}

\subsubsection*{Notation}

In the following, $k$ denotes a field of characteristic $p$ (prime or $0$)
with algebraic closure $K$. Hyperelliptic curves are additionally assumed to
be smooth, so that when a singular affine model of a curve is given, we
actually consider its desingularization.  Unless noted otherwise,
(iso)morphisms are defined over the base field $k$.  We use the following
notation for groups: $\CG_n=\ZZ/n\ZZ$; $\DG_{2n}$ is the dihedral group with
$2n$ elements; $\UG_6$ is the group with $24$ elements defined by 
$\langle S,T\rangle$ with $S^{12}=T^2=1$ and $TST=S^5$; $\VG_8$ is the group with $32$
elements defined by $\langle S,T\rangle$ with $S^4=T^8=(ST)^2=(S^{-1}T)^2=1;$
$\SG_n$ is the symmetric group over $n$ symbols. Finally, 
if $f_1$ and $f_2$ are polynomials or matrices or some other such 
objects over a field~$k$, we will write
$f_1\sim f_2$ if there exists $\lambda \in k^*$ such that $f_1=\lambda \cdot f_2$.

\section{Isomorphisms between forms and hyperelliptic curves}

\subsection{Isomorphisms of binary forms}

Let $n \geq 1$ be an integer, let $V=k^2$ be the $k$-vector space
with basis $(x,z)$, and let $S^n(V)$ be the $(n+1)$-dimensional
vector space of homogeneous forms $\sum_{i=0}^n a_i x^i z^{n-i}$
of degree $n$ in $(x,z)$. In the sequel, we call an element of
$S^n(V)$ a \emph{\textup{(}binary\textup{)} form}. When $n=0$, we let $S^0(V)=k$.
Let $G$ be a subgroup of $\GL_2(k)$ and let $M$ be an element of~$G$.
If $f$ is a form in $S^n(V)$, we define $M.f$ by 
$(M.f) (x,z)=f(M^{-1} (x,z))$, where the action of a matrix on
$(x,z)$ is the standard action on ${^t (x,z)}$.

\begin{definition}
Let $f_1,f_2$ be forms of degree $n \geq 1$ over a field $k$.
We denote by $\Isom(f_1,f_2) \subset \PGL_2(k)$ the set of matrices
$M$ up to scalar equivalence such that $M. f_1\sim f_2$. 
Additionally, we write $\Aut f_1$ for $\Isom(f_1,f_1)$.
\end{definition}

If $\Isom(f_1,f_2) \ne \emptyset$, this set is a principal
homogeneous space over $\Aut f_1 $. In particular,
$\Isom(f_1,f_2)=M  \Aut f_1 $ for any $M \in \Isom(f_1,f_2)$.

Let $f$ be a form of degree $n$ over~$k$. Over $K$,
we can write $f =\prod_{i=1}^s (\alpha_i x-\beta_i z)^{n_i}$, where 
$(\alpha_i,\beta_i) \in K^2 \setminus \{(0,0)\}$ and $n_i \in \NN$. 
We associate to such a form its
squarefree part $\tilde{f} = \prod_{i=1}^s (\alpha_i x-\beta_i z)$,
which is defined up to a multiplicative constant.  
The action of $M$ on $f$ reflects the classical M\"obius action
of $\PGL_2(K)$ on the roots $(\alpha_i : \beta_i) \in \PP_K^1$ of~$f$.
In particular, two forms of the same degree are $K$-isomorphic if and only
if there exists an $M \in \GL_2(K)$ mapping the roots of the first form 
to the roots of the second form (counting multiplicities). Hence we have:

\begin{lemma}
The group $\Aut_K f $ is finite if and only if $s \geq 3$, that is,
if and only if $\deg(\tilde{f}) \geq 3$. 
Moreover, $\Aut_K f  \subset \Aut_K \tilde{f}.$
\end{lemma}

\subsection{The direct approach}
\label{sec:classical-approach}

The classical method to compute isomorphisms between two binary forms 
$f_1$, $f_2$ of degree $n$ over a field $k$ is to find a 
$\PGL_2(k)$-transformation of $\PP^1$ which maps the roots of the 
first form to the root of the second form.  The most time-consuming
task is to compute an isomorphism between the splitting fields of
$f_1$ and~$f_2$.  Even in the most favorable case, where
$k$ is a finite field, the fastest algorithms need at least
$O(n^{2,5+o(1)})$ operations in $k$ (see~\cite{\KedlayaUmans}).

We show here that it is actually possible to get rid of this
cumbersome ring isomorphism computation, and describe an
algorithm of time complexity only quasilinear in~$n$. This algorithm
takes as input binary forms $f_1 = \sum_i A_i x^i z^{n-i} $ and 
$f_2 = \sum_i B_i x^i z^{n-i}$ of equal degree $n \geq 3$,
each having at least three distinct roots.  
It returns matrices representing the elements of $\Isom(f_1,f_2)$.

First, we suppose that the coefficient 
$A_{n-1}$ is equal to zero. Note that this is
typically not a big restriction, since we may apply
linear transformations to~$f_1$.  A notable exception
is when $p$ divides~$n$. We therefore assume that $p$ is prime to $n$.

Second, we note that determining $\Isom(f_1,f_2)$ is equivalent to determining the
matrices $M = (m_{i,j}) \in \GL_2 (k)$ such that
\begin{equation}
\label{eq:Meq}
  f_2 (m_{11} x + m_{12} z, m_{21} x + m_{22} z) 
   = \lambda f_1 (x,z) \text{ \quad for some } \lambda \in k^*.
\end{equation}

Third, because of homogeneity, we may suppose that the $\lambda$ in
Equation~\eqref{eq:Meq} equals $1$, after enlarging $k$ by a radical extension if
necessary.  Note that though this radical extension is \emph{a priori} unknown, the
details of the algorithm below will show how it can be determined.

Finally, we may suppose that the $M$ in Equation~\eqref{eq:Meq} are of the form
$$
  M = \begin{bmatrix}
               1/\alpha & \beta/\delta \\
          \gamma/\alpha &     1/\delta 
      \end{bmatrix}.
$$
Of course this may not be true, because a zero may occur on the diagonal of
one of these $M$.  However, one can fix this situation by applying a suitable
change of variables to~$f_2$.

The equation $f_2(m_{11} x + m_{12} z, m_{21} x + m_{22} z) = f_1(x, z)$
now becomes
$$f_2(x + \beta z, \gamma x + z) = f_1(\alpha x, \delta z).$$
Equating the coefficients of $x^n$ in both sides of this equation yields
$A_n  \alpha^n = f_2(1, \gamma),$
and we can write $\alpha^n$ in terms of $\gamma$.  Similarly, the
equality of the coefficients of $x^{n-1} z$,
$$
  \beta   \frac{\partial f_2}{\partial x}(1, \gamma)
    +     \frac{\partial f_2}{\partial z}(1, \gamma)  = 0,
$$
enables us to write $\beta$ in term of $\gamma$ too. 
More generally, equating the coefficients of $x^{n-i} z^i$
for $i=2,\ldots, n$, where we substitute $\alpha^n$ and $\beta$ 
in term of $\gamma$, yields $n-1$ equations of the form\
\begin{multline}
\label{eq:4}
A_n \biggl( \sum_{j=0}^{i} \binom {i} {j}
      \left( -\frac {\partial   f_2}{\partial z}\right)^j 
      \left(  \frac {\partial   f_2}{\partial x}\right)^{i-j} 
              \frac {\partial^i f_2}{\partial x^j \partial{z}^{i-j}} \biggr)
              (1, \gamma) \\
 = i! \left(\frac {\partial f_2}{\partial x}(1, \gamma)\right)^{i}
      \left(\frac {\delta}{\alpha}\right)^i
             f_2(1, \gamma) .
\end{multline}
Note that the left hand side of Equation~\eqref{eq:4} is actually a 
polynomial multiple of $f_2(x, z)$, and we can divide both sides by 
$f_2(1,\gamma)$ --- see~\cite[Chapter 1, \S\S15--16]{\graceyoung} 
for an elegant explanation.  This yields equations of degree $i(n-2)$ 
in $\gamma$ for the left side and of degree $i(n-1)$ in $\gamma$ and 
degree $i$ in $\delta/\alpha$ on the right side.

Now, dividing the square of Equation~\eqref{eq:4} specialized at $i=3$ 
by the cube of Equation~\eqref{eq:4} specialized at $i=2$ allows to 
eliminate, up to some constant, the right hand side of these equations, 
in particular the unknown $\delta/\alpha$. We end up with an equation 
of degree $6(n-2)$ in $\gamma$.  Similarly, when $n>3$, dividing 
Equation~\eqref{eq:4} specialized at $i=4$ by the square of 
Equation~\eqref{eq:4} specialized at $i=2$ yields an equation of degree 
$4(n-2)$ in $\gamma$. Taking the $\gcd$, we obtain a polynomial of low 
degree with root $\gamma$. Generically, this $\gcd$ is of degree $1$.

Under the assumptions made, the algorithm is therefore straightforward.  For
each possible $\gamma$, we compute $\alpha,\beta$ and $\delta$ and check
whether the resulting matrix is in $\Isom(f_1,f_2)$.

The computations involved in this algorithm (taking $\gcd$s of polynomials of
degree $O(n)$, taking $n$-th roots, and so forth) are all of time complexity
quasilinear in~$n$.

We have implemented the algorithm in Magma (version 2.18-2)
and have timed the resulting procedure, \texttt{IsGL2EquivFast}, 
on a laptop (based on an Intel Core i7 M620 2.67GHz processor)
for irreducible forms of increasing degree, the most favorable case 
for the native Magma routine \texttt{IsGL2Equivalent}.
We compare with \texttt{IsGL2Equivalent}, which implements the
classical method, first over the finite field $\FF_{10007}$, 
then over the rationals with coefficients bounded by $\pm 2$. 
The results are in Table~\ref{tab:isgl2equiv}.
(See Section~\ref{sec:gener-exampl-genus}
for the definition of \texttt{IsGL2EquivCovariant}.)

\begin{table}[htbp]
\begin{center}
\begin{tabular}{rlcr@{}lr@{}lr@{}lccr@{}lr@{}lr@{}l}
\toprule
&&&\multicolumn{7}{c}{Computations over $\FF_{10007}$}&&\multicolumn{6}{c}{Computations over $\QQ$}\\
\cmidrule(lr){4-10}\cmidrule(lr){12-17}
\multicolumn{2}{c}{Genus}&\hbox to 1em{}& \multicolumn{2}{r}{Old} & \multicolumn{2}{r}{\S\ref{sec:classical-approach}} & \multicolumn{2}{r}{\S\ref{sec:gener-exampl-genus}}&
                         &              & \multicolumn{2}{r}{Old} & \multicolumn{2}{r}{\S\ref{sec:classical-approach}} & \multicolumn{2}{r}{\S\ref{sec:gener-exampl-genus}}\\
\midrule
   1 &&&        0 & .0 &   0 & .0 & \ 0 & .0 &&&    0 & .0 &    0 & .0 & \ 0 & .0 \\
   2 &&&        0 & .0 &   0 & .0 &   0 & .0 &&&    0 & .0 &    0 & .0 &   0 & .0 \\
   4 &&&        0 & .0 &   0 & .0 &   0 & .0 &&&    0 & .4 &    0 & .0 &   0 & .0 \\
   8 &&&        0 & .0 &   0 & .0 &   0 & .0 &&&   15 &    &    0 & .0 &   0 & .0 \\
  16 &&&        0 & .1 &   0 & .0 &   0 & .0 &&& 1150 &    &    0 & .1 &   0 & .0 \\
  32 &&&        0 & .2 &   0 & .0 &   0 & .0 &&& ---  &    &    0 & .2 &   0 & .0 \\
  64 &&&        0 & .9 &   0 & .1 &   0 & .0 &&& ---  &    &    0 & .6 &   0 & .0 \\
 128 &&&        6 & .5 &   0 & .6 &   0 & .0 &&& ---  &    &    3 &    &   0 & .2 \\
 256 &&&       39 &    &   3 & .7 &   0 & .1 &&& ---  &    &   30 &    &   0 & .6 \\
 512 &&&      242 &    &  25 &    &   0 & .5 &&& ---  &    &  382 &    &   3 & .4 \\
1024 &&& \ \ 1560 &    & 165 &    &   2 & .5 &&& ---  &    & 5850 &    &   7 &    \\
\bottomrule
\end{tabular}
\bigskip
\caption{Timings (in seconds) for isomorphisms between forms of degree $2g+2$,
over $\FF_{10007}$ and over $\QQ$.
The columns labeled `Old' give timings for Magma's built-in function
\texttt{IsGL2Equivalent}; the columns labeled `\S\ref{sec:classical-approach}'
give timings for the function \texttt{IsGL2EquivFast} described in
Section \ref{sec:classical-approach}; and the columns labeled 
`\S\ref{sec:gener-exampl-genus}' give timings for the function 
\texttt{IsGL2EquivCovariant} described in Section \ref{sec:gener-exampl-genus}.
Entries of `---' indicate computations that were aborted after an hour.
}
\label{tab:isgl2equiv}
\end{center}
\end{table}

As concluding remarks, we note first of all that this algorithm 
is equally suitable for determining $K$-isomorphisms. 
Moreover, in the special case of binary quartics,
it is just as efficient as the algorithm given in~\cite{\CremonaFisher}.

\subsection{The covariant approach}
\label{sec:other-geom-covar}

Let $k$ be an infinite field of characteristic $p$ and 
let $n>1$ be an integer.

\begin{definition} 
Let $r \geq 0$ be an integer. A  homogeneous polynomial 
function $C\colon S^n(V) \to S^r(V)$ of degree $d$ is 
a \emph{covariant} if there exists $\omega \in \ZZ$ 
such that for all $M \in G$ and all $f \in S^n(V)$, we have 
$$C(M.f)= (\det M)^{-\omega} \cdot M.C(f).$$
When $r=0$, such a $C$ is called a (relative)
\emph{invariant} and is denoted by~$I$.
\end{definition}

The integer $r$ is called the \emph{order} of the covariant. 
If $nd-r$ is odd, the covariant is necessarily zero. 
Otherwise the integer $\omega$ is unique,
and is called the \emph{weight} of the covariant.
It is equal to $(nd-r)/2$. In the sequel, we often 
identify $C$ with $C(f)$ for a general form $f \in F(a_0,\ldots,a_n)[x,z]$,
where $F$ is the prime field of $k$.  For instance, the identity 
function $S^n(V) \to S^n(V)$ is a covariant of degree $1$ and 
order $n$ that we identify with $f$ itself.

\begin{remark}
The determinant factor prevents the addition of covariants of different
weights when $G=\GL_2(K)$. Hence one generally studies the graded
algebra $\cc_n$ of covariants and $\ii_n$ of invariants under the 
action of $\SL_2(K)$. It is easy to see that the homogeneous elements
of $\cc_n$ and $\ii_n$ are actually all the covariants or invariants
under the action of $\GL_2(K)$.  Despite this ambiguity, in the rest
of the article we work with $G=\GL_2(K)$ instead of $\SL_2(K)$ because,
in practice, this choice often allows us to avoid a quadratic extension
of $k$ when looking for an isomorphism $M$ between two forms.
\end{remark}

There is a large literature on how to generate invariants
and covariants starting from~$f$. Gordan's algorithm \cite{\gordan}
allows to find a set of generators for the algebras $\cc_n$ and 
$\ii_n$ thanks to the use of certain differential operators, called
\emph{$h$-transvectants} and defined as follows.
Given two covariants $C_1, C_2$ of degree $d_1, d_2$ and 
of order $r_1, r_2$, and given an integer $h \geq 1$,
we can create a new covariant denoted $(C_1,C_2)_h$
and usually defined as \cite[p.~88]{\olver}
$$\frac{(r_1-h)!(r_2-h)!}{r_1!r_2!} 
   \sum_{i=0}^h\ (-1)^{i}\ \binom{h}{i}\ \frac{\partial^{h} C_1}{\partial x^{h-i} \partial z^i}
                                       \ \frac{\partial^{h} C_2}{\partial x^i\partial z^{h-i}}.$$
In practice, we use the univariate counterpart.
Looking at $C_1$, $C_2$ as univariate polynomials in $x/z$,
we get \cite[Theorem~5.6]{\olver}
\begin{equation} 
\label{eq:transvectant}
h! \frac{ (r_1-h)! (r_2-h)!}{r_1! r_2!}
  \sum_{i=0}^h\ (-1)^{i}\ \binom{r_1-i}{h-i}\ \binom{r_2 - h+i}{i}
                        \ \frac{d^{h-i} C_1}{d x^{h-i}}\ \frac{d^i C_2}{d x^i}.
\end{equation}
Effective methods for computing sets of generators when $K=\CC$
have been worked out for $n$ up to $10$ (see 
\cite{\dixmier,\vongall,\dixmierlazard,\bedradyuk,\shioda,\cronih,\brouwerone,\brouwertwo}).
It has been shown that if $\CC$ is replaced by 
an algebraically closed field $K$ of characteristic $p$,
these computations are still valid for $g=2$ if $p \ne 2,3,5$ \cite{\LReight} 
and for $g=3$ if $p \ne 2,3,5,7$ \cite{\LReleven}.

Our second idea to compute isomorphisms between forms of a given degree is to
reduce the question to smaller degree by using covariants. Indeed, the
following observation is a simple consequence of the definition
itself.

\begin{proposition}
\label{prop:ftoc}
Let $f_1,f_2$ be forms of even degree
$n$ over a field $k$.  Let $C$ be a covariant of order $r$ for binary forms
of degree $n$, defined over the prime field of $k$, and let $c_i=C(f_i) \in
S^r(V)$.  Then $\Isom(f_1,f_2) \subset \Isom(c_1,c_2).$ \qed
\end{proposition}

We illustrate this idea and study its limitations with the
computation of isomorphisms for forms and hyperelliptic curves 
in Sections~\ref{sec:gener-exampl-genus} and~\ref{sec:applications}.
As we want the covariants $c_i$ to have the smallest degree
possible and $\Isom(c_1,c_2)$ to be finite, we want that 
$\deg(\tilde{c_i}) \geq 3$. 
Actually, in what follows we mostly deal with forms of even degree,
so nonzero covariants will be of even order, and the smallest 
degree meeting our restriction is then $4$.

Consider a binary quartic
$q = a_4 x^4 + a_3 x^3 z + a_2 x^2 z^2 + a_1 x z^3 + a_0 z^4$
over $k$ with $p \ne 2,3$. We define
\begin{alignat*}{2}
  I&=I(q) &&= 12 a_4 a_0  - 3 a_3 a_1 + a_2^2 \\
  J&=J(q) &&= 72a_4 a_2 a_0 + 9 a_3 a_2 a_1- 27 a_4 a_1^2 - 27 a_0 a_3^2 - 2a_2^3 
\end{alignat*}
as in \cite{\CremonaFisher}. The form $q$ has distinct roots if and only if
$\Delta=4 I^3 - J^2 \ne 0$.  Given $I,J \in K$ such that $\Delta \ne 0$, one
can easily reconstruct a form with at least three distinct roots which is
$K$-isomorphic to $q$. We can take
\begin{equation}
\label{quarticsplit}
q = \begin{cases}
     x^3 z - 27 ({I^3}/{J^2}) x z^3 - 27 ({I^3}/{J^2})  z^4 & \quad \text{ if }\ J\ne 0,\\
     x^3 z+x z^3 & \quad \text{ otherwise.}
\end{cases}     
\end{equation}

Concerning the geometric automorphisms of binary quartics, we have the
following easy result, for which we could not find a reference.

\begin{proposition}
\label{prop:autquartics}
Let $q$ be a binary quartic form over $K$, with invariants $I$ and~$J$.
Suppose that $\Delta \ne 0$. Then
\begin{equation}
\label{binquadaut}
  \Aut q \cong  \begin{cases}
                \AG_4 & \text{if $I = 0$},\\
                \DG_8 & \text{if $J = 0$},\\
                \DG_4 & \text{otherwise.}
                \end{cases}
\end{equation}
\end{proposition}

\begin{proof}
Let $\Lambda \subset \PP^1 (K)$ be the set of four
roots of~$q$.  Using the $3$-transitivity of the action 
of $\PGL_2 (K)$ on $\PP^1 (K)$, we may assume that 
$\Lambda = \{ 0, 1 , \infty , \lambda \}$ for some 
$\lambda \in K \setminus \{0, 1 \}$.
Then the transformation $x \mapsto \lambda / x$
induces the permutation $(0 \infty) (1 \lambda)$ of~$\Lambda$.
By symmetry, we see that $\Stab \Lambda \subset \Sym \Lambda$
contains the Viergruppe $\DG_4 \subset \Sym \Lambda$.

We are reduced to analyzing the case when 
$\Stab \Lambda$ properly contains $\DG_4$. 
Since the extension 
$1 \rightarrow \DG_4 \rightarrow \SG_4 \rightarrow \SG_3 \rightarrow 1$
is split and all subgroups of $\SG_3$ of equal order are conjugate, 
this is in turn equivalent to determining when
$\Stab \Lambda$ contains an additional given $2$- or
$3$-cycle. These cases give rise to the exceptional groups
in Equation~\eqref{binquadaut} of order $8$ and~$12$.

First let us see for which $\lambda$ the permutation
$(1 \lambda)$ is in $\Stab \Lambda$. In this case, 
the fractional linear transformation fixes $0$ and $\infty$
and is therefore of the form $x \mapsto c x$. This only 
gives a new automorphism if $c = -1$, so $\lambda = -1$ and $J = 0$.

In the case where the permutation $(0 1 \lambda)$ is in 
$\Stab \Lambda$, a slightly more involved calculation gives 
that $\lambda = \zeta_3 + 1$ for a primitive third root of unity
$\zeta_3$, and in that case $I = 0$.
\end{proof}

We will also need in the sequel the following result.

\begin{proposition}
\label{prop:iso-quartic}
Let $q$ be a binary quartic form defined over $k$ with distinct roots,
and let $\frakq$ be the form defined by Equation~\eqref{quarticsplit}.
Assume that $I(q) \ne 0$ and $J(q) \ne 0$.  Then a $K$-isomorphism
between $q$ and $\frakq = z (x^3+ b_1 x z^2 + b_0 z^3)$ is 
defined over any extension of $k$ where $q$ has a root.
\end{proposition}

\begin{proof}
Let $k'$ be an extension of $k$ where $q$ has a root. 
By a change of variable defined over $k'$, we can map this
root to infinity and hence $q$ onto  $q'=z r$,
where $r=x^3 + a_1 x z^2+ a_0 z^3 \in k'[x,z]$. 
Now since 
\begin{align*}
I(q')&= -a_1/4 & I(\frakq)&= -b_1/4 \\
J(q')&=-a_0/16 & J(\frakq)&=-b_0/16
\end{align*}
we get the relation $a_1^3/a_0^2 = b_1^3/b_0^2$. 
Hence if we define $\lambda\in k'$ by
$$\lambda = \frac{J(q') I(\frakq)}{J(\frakq) I(q')},$$
the $k'$-isomorphism $M\colon (x,z) \mapsto (\lambda x,z)$ maps 
$q'$ onto $\frakq$.
\end{proof}

\subsection{Generic forms of even degree}
\label{sec:gener-exampl-genus}

We now describe an algorithm, based on the ideas of
Sections~\ref{sec:classical-approach} and \ref{sec:other-geom-covar},
to compute the isomorphisms between two generic binary forms 
$f_1$ and~$f_2$. Our notation is as in~\ref{sec:classical-approach}.

\begin{algorithm}[\texttt{IsGL2EquivCovariant}]
\label{alg:covariant}
\begin{alglist}
\algin  Two forms $f_1$ and $f_2$ of the same degree $n \geq 3$ over $k$,
        and integer parameters $\Bord\ge3$, $\Bdeg\ge2$, and $\Bsing\ge0$.
\algout The matrices $M=(m_{i,j})_{i,j}$ in $\PGL_2(k)$ such that $M.f_1 \sim f_2$.
\item \emph{Order loop}.
      For $o$ increasing from $3$ to $\Bord$ do:
      \begin{algsublist}
      \item \emph{Degree loop}.
            For $d$ increasing from $2$ to $\Bdeg$ do:
            \begin{algsubsublist}
            \item Compute a random covariant $C$ of order $o$ and degree $d$ using
                  transvectants.
            \item If $\tilde{C}(f_1)$ is of degree at least 3, then compute
                  $\Isom(\tilde{C}(f_1), \tilde{C}(f_2))$ and return the elements which
                  induce isomorphisms between $f_1$ and $f_2$.
            \item Otherwise, repeat the following procedure $\Bsing$ times:
                  \begin{itemize}
                  \item[--] Compute a new random covariant $C'$ of order $o$ and degree $d$
                        using transvectants,
                        and replace $C$ by the covariant $C + \kappa C'$ for some random
                        $\kappa$ in the field $k$. 
                  \item[--] If $\tilde{C}(f_1)$ is of degree at least $3$,
                        compute
                        $\Isom(\tilde{C}(f_1),\tilde{C}(f_2))$ and return the elements that
                        induce isomorphisms between $f_1$ and $f_2$. 
                  \end{itemize}
            \end{algsubsublist}
      \end{algsublist}
\item \emph{Failure}. Return the result of \texttt{IsGL2EquivFast}($f_1$, $f_2$).
\end{alglist}
\end{algorithm}

For the purpose of computing random covariants, we follow
Gordan~\cite{\gordan}. Given an order $o$ and a degree $d$, we construct
recursively a covariant $C = \left( \prod C_{d',o' },\ f \right)_h$ as a
transvectant of some level $h$ of the form $f$ and a product of covariants of
intermediate orders $o'$ and degrees $d'$, under the two constraints 
$d = \sum d'$ and $o = n + \sum o' -2 h$.

When $n$ is even, the transvectant of smallest order and degree is
$C_{2,4} = (f,f)_{n-2}$. The next simplest transvectant is 
$C_{3,4} = ((f, f)_{n/2}, f)_{n-2}$, of order $4$ and degree $3$.
For large orders and degrees, covariants must be computed `on the fly',
specialized for $f_1$ and $f_2$, since expressions are far too large 
to be precomputed.

To completely specify the algorithm, we have to be more precise 
about how to compute covariants and how to choose the loop bounds
$\Bord$, $\Bdeg$ and $\Bsing$.
A straightforward choice for the loop bounds 
is $\Bord=4$, $\Bdeg=2$, and $\Bsing=0$. With this choice, 
only the covariant $C_{2, 4} = (f, f)_{n-2}$ is tested for $n$ even, 
and when it turns out that the discriminant of this covariant vanishes,
we go back to the method \texttt{IsGL2EquivFast}.  
First note that the covariant $( f, f)_{n-2}$ can be easily computed.
Using Equation~\eqref{eq:transvectant}, we find that we can write
\begin{equation}
\label{eq:ff24}
\frac { \left( n!  \right)^2}{ \left( n-2 \right) !}  ({ f}, { f})_{n-2}
  = c_4 x^4 + c_3 x^3 z + c_2 x^2 z^2 + c_1 x z^3 + c_0 z^4,
\end{equation}  
where the coefficients $c_i$ are given by
{\allowdisplaybreaks
\begin{align*}
   c_0 &= \sum _{k=0}^{n-2} (-1)^k (n-k)!\, (k+2)!\, a_{n-2-k}  a_k \\
   c_1 &= \sum _{k=0}^{n-2} (-1)^k (n-k)!\, (k+2)!\, 
            \bigl( (n-1-k) a_{n-1-k}  a_k + (k+1) a_{n-2-k} a_{k+1} \bigr) \\
   c_2 &= \frac{1}{2} \sum _{k=0}^{n-2} (-1)^k (n-k)!\, (k+2)!\, \bigl( (k+2)(k+1) a_{k+2} a_{n-2-k} \\
       & \qquad\qquad\qquad    + 2(n-1-k)(k+1) a_{k+1} a_{n-1-k} + (n-k)(n-1-k) a_k  a_{n-k} \bigr) \\
   c_3 &= \sum _{k=0}^{n-2} (-1)^k (n-k)!\, (k+2)!\, 
            \bigl( (n-1-k) a_{n-k} a_{k+1}  
               + (k+1) a_{n-1-k} a_{k+2} \bigr) \\
   c_4 &= \sum _{k=0}^{n-2} (-1)^k (n-k)!\, (k+2)!\, a_{n-k} a_{k+2} .
\end{align*}
Moreover, this setting is a good option for generic forms, as the following
proposition shows.
}

\begin{proposition}
Let $n \geq 6$ be an even integer and $p \neq 2,3$. Let $f$ be a generic
binary form of degree $n$ over $k$. Then the discriminant of $C_{2,4} (f)$
is nonzero.
\end{proposition}

\begin{proof}
It is enough to find a single form $f$ of degree $n$ for which
$C_{2,4}(f)$ has nonzero discriminant. First let us suppose that
$p$ is coprime to $n (n-2) (n-3) (n^2+3 n+6)$. We then take 
$f = x^n + x^{n-1} z - x z^{n-1} - z^n$. Note that this form is in fact
nonsingular because $f = (x + z)(x^{n-1} - z^{n-1})$. We have that 
$$-C_{2,4} (f) = \frac{4}{n}  x^3 z 
                +  \frac{2(n^2 - n + 6)}{n^2} x^2 z 
                + \frac{4}{n}  x z^2.$$
This form has discriminant equal to $64 (n-3) (n-2) (n^2+3 n+6)/n^6$,
which is nonzero by hypothesis.
  
One calculates similarly that for the other values 
of $p \neq 2,3,5$, one can use the form 
$x^n + x^{n-1} z + x z^{n-1} - z^n$ instead. Indeed,
under these hypotheses on $p$ the numerator
$n^4 + 2 n^3 + 5 n^2 - 12 n + 36$  of the resulting discriminant
is coprime to the previous numerator. To finish the proof, $p=5$ 
can be excluded using the form 
$x^n + x^{n-1} z + x z^{n-1} + 2 z^n$.
\end{proof}
 
For nonrandom forms, especially forms of small degree with 
nontrivial automorphism group, it may be interesting to test
other covariants than merely $C_{4,2}$.  We then propose the 
following settings:
$$  \Bord=\min(8, n),\ \Bdeg=10,\ \text{ and }\ \Bsing=10. $$
These bounds are constant in order to keep the total time complexity
quasi-linear in~$n$. More precisely, the bound $\Bord$ is chosen
to be at most $8$ so as to take advantage of the classification work
of~\cite{\LReleven}, the bound $\Bdeg$ is chosen to cover all the
possible fundamental covariants of degree $8$ and with order 
between $4$ and~$8$
(see~\cite[Table~1, p.~607]{\LReleven}),
and the bound $\Bsing$ is chosen so as to increase the
probability that our covariants, if singular, have distinct
points of singularity (so that a linear combination may be
nonsingular).

\begin{remark}
We may enter the last loop of the algorithm even 
if the form $f$ has no geometric automorphisms.
For example, this happens with the degree-$8$ form
$$x^7 z + 7 x^6 z^2 + 7 x^5 z^3 + 8 x^4 z^4 + 2 x^3 z^5 + 10 x^2 z^6 + 9 x z^7$$
over $k = \FF_{11}$. 
\end{remark}

We have programmed Algorithm~\ref{alg:covariant}
in Magma (version 2.18-2), using the first setting
of the parameters.
In particular, we have implemented the covariant $C_{4,2}$
using Equation~\eqref{eq:ff24}, and we have measured the 
timings of the resulting procedure, \texttt{IsGL2EquivCovariant},
in the same experiments as in Section~\ref{sec:classical-approach}.
The results are presented in Table~\ref{tab:isgl2equiv}.
As expected, computing isomorphisms is much faster with the help
of covariants, even if the forms are split over~$k$.

\subsection{Application to isomorphisms of hyperelliptic curves}
\label{sec:applications}

\subsubsection{Isomorphisms of forms and of hyperelliptic curves}
\label{sec:iso-general-fac}

A curve $\Xm$ of genus $g \ge 1$ defined over $k$ will 
be called \emph{hyperelliptic} if $\Xm / K$ has a 
separable degree-$2$ map to~$\PP^1_{K}$.  If $g > 1$, 
the curve $\Xm$ then has a unique involution~$\iota$,
called the \emph{hyperelliptic involution}, such that 
$\Qm=\Xm/\langle \iota \rangle$ is of genus~$0$. 
This involution is in the center of $\Aut_{K} \Xm$. 
We call $\Autred_K \Xm =(\Aut_{K} \Xm) / \langle \iota \rangle$
the \emph{reduced automorphism group} of $\Xm$.

Let us assume from now on that $p \ne 2$.  Then if $Q$ 
has a rational point, $\Xm$ is birationally equivalent
to an affine curve of the form $y^2= f(x)$ for a separable
polynomial $f$ of degree $2g+1$ or $2g+2$.  We say that
$f$ is a \emph{hyperelliptic polynomial} and that $\Xm$
has a \emph{hyperelliptic equation} if a curve in 
its isomorphism class (over~$k$) can be written in 
the form above.  We denote by $\Xm_f$ the curve
associated to a hyperelliptic polynomial~$f$. A hyperelliptic
curve automatically has a hyperelliptic equation when $k$
is algebraically closed or a finite field. However, for
more general fields and curves of odd genus, this is not
necessarily the case (see~\cite{\LReleven}).

By homogenizing to weighted projective coordinates
of weight $(1,g+1,1)$, we obtain an equation 
$y^2 = f(x,z)$. Here $f$ is seen as a form of 
degree $2g+2$, taking into account a `root' at 
infinity when $\deg f=2g+1$.  With this convention,
the roots of $f$ are the ramification points of
the cover $\Xm/Q$. We will use these conventions 
for the roots and degree in the sequel when we 
speak about a hyperelliptic polynomial or the associated form.

If $f_1$ and $f_2$ are hyperelliptic polynomials of 
even degree $2g+2 \geq 6$, then isomorphisms between the
hyperelliptic curves $y^2 = f_i(x,z)$ are represented by pairs
$(M,e)$ with 
$$M =\begin{bmatrix}
     a & b \\ 
     c & d 
     \end{bmatrix} \in \GL_2(k)$$
and  $e \in k^*$. To such a couple, one
associates the isomorphism 
$$(x,z,y) \mapsto (ax+bz,cx+dz, ey).$$
The representation is unique up to the equivalence
$(M,e) \equiv (\lambda M,\lambda^{g+1} e)$ for 
$\lambda \in k^*$. Hence, if $M.f_1=\mu \cdot f_2$
then the map
\begin{align*}
\Isom(f_1,f_2)  &    \to  (\GL_2(k) \times K^*)/\equiv  \\
M               &\mapsto  (M,\pm \sqrt{\mu})
\end{align*}
is well-defined up to the choice of a sign. It surjects onto 
$\Isom(\Xm_{f_1},\Xm_{f_2})$, so knowing
$\Isom(f_1,f_2)$ is enough to determine 
$\Isom(\Xm_{f_1},\Xm_{f_2})$
`up to the hyperelliptic involution'.

\subsubsection{Hyperelliptic curves of genus $2$ and $3$}
\label{sec:all-hyper-genus}

The covariant approach requires a covariant with at least
three distinct roots, and hence it may fail in special cases,
which we can specify for small genera. We give some details
on the more difficult of the two cases: the genus-$3$ case.
This problem is naturally stratified by the possible
automorphism groups of the curve; we list these
automorphism groups, together with normal models and inclusion
relations between the strata, in Figures~\ref{tab:auto}
and~\ref{tab:auto2}.  We assume here that $p=0$ or $p>7$.

\begin{figure}[htbp]
\begin{center}
\begin{tabular}[t]{rlcrlcl}
\toprule
\multicolumn{2}{c}{$\Aut_K \Xm_f$} && \multicolumn{2}{c}{$\Autred_K \Xm_f$} && Normal models $\Xm_f\colon y^2=f$ \\
\midrule
& $\CG_2$              &&\hbox to 1ex{}& $\{1\}$    && $f = x (x-1) (x^5 + ax^4 + bx^3 + cx^2 + dx + e)$ \\[1ex]
& $\DG_4$              &&              & $\CG_2$    && $f = x^8 + ax^6 + bx^4 + cx^2 + 1$ or             \\
&                      &&              &            && $f = (x^2 - 1) (x^6 + ax^4 + bx^2 + c)$           \\[1ex]
& $\CG_4$              &&              & $\CG_2$    && $f = x (x^2 - 1) (x^4 + ax^2 + b)$                \\[1ex]
& $\CG_2^3$            &&              & $\DG_4$    && $f = (x^4 + ax^2 + 1) (x^4 + bx^2 + 1)$           \\[1ex]
& $\CG_2 \times \CG_4$ &&              & $\DG_4$    && $f = (x^4 - 1) (x^4 + ax^2 + 1)$ or               \\
&                      &&              &            && $f = x (x^2 - 1) (x^4 + ax^2 + 1)$                \\[1ex]
& $\DG_{12}$           &&              & $\DG_6$    && $f = x (x^6 + ax^3 + 1)$                          \\[1ex]
& $\CG_2 \times \DG_8$ &&              & $\DG_8$    && $f = x^8 + ax^4 + 1$                              \\[1ex]
& $\CG_{14}$           &&              & $\CG_7$    && $f = x^7 - 1$                                     \\[1ex]
& $\UG_6$              &&              & $\DG_{12}$ && $f = x (x^6 - 1)$                                 \\[1ex]
& $\VG_8$              &&              & $\DG_{16}$ && $f = x^8 - 1$                                     \\[1ex]
& $\CG_2 \times \SG_4$ &&              & $\SG_4$    && $f = x^8 + 14x^4 + 1$                             \\
\bottomrule
\end{tabular}
\bigskip
\caption{Automorphism groups of genus-$3$ hyperelliptic curves.
For each automorphism group, we list the associated reduced
automorphism group, together with normal model(s) for the generic
hyperelliptic curve with that automorphism group.  The notation for the
groups is given at the end of the Introduction.}
\label{tab:auto}
\end{center}
\end{figure}

\begin{figure}[htbp]
$$  \xymatrixrowsep{18pt}
%   \xymatrixcolsep{2pt}
\xymatrix{
         & \CG_2 \ar@{-}[ddddl] \ar@{-}[dd] \ar@{-}[rd]  &                                            &                                           && \text{$5$-dimensional}\\
         &                                               & \DG_4 \ar@{-}[dd] \ar@{-}[ddl] \ar@{-}[dr] &                                           && \text{$3$-dimensional}\\
         & \CG_4 \ar@{-}[d]                              &                                            & \CG_2^3 \ar@{-}[d]                        && \text{$2$-dimensional}\\
         & \CG_2 \times \CG_4 \ar@{-}[d]  \ar@{-}[dr]    & \DG_{12} \ar@{-}[dl] \ar@{-}[dr]           & \CG_2 \times \DG_8 \ar@{-}[dl] \ar@{-}[d] && \text{$1$-dimensional}\\
\CG_{14} & \UG_6                                         & \VG_8                                      & \CG_2 \times \SG_4                        && \text{$0$-dimensional}\\
}
$$
\caption{Dimensions and containment relationships among the moduli spaces of
genus-$3$ hyperelliptic curves with given automorphism groups.}
\label{tab:auto2}
\end{figure}
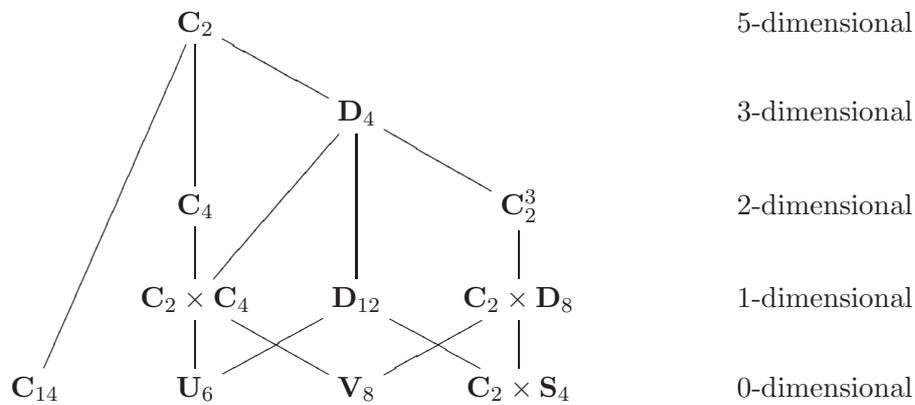

The moduli space of hyperelliptic curves of genus $3$
is $5$-dimensional, and can be explicitly described
using the Shioda invariants $J_2, J_3, \ldots, J_{10}$
constructed in \cite{\shioda}.  These invariants were used
to speed up the calculations leading to the proof of the 
following proposition, which shows that the locus where 
the covariant method fails is of codimension $4$ in the
full moduli space.  (The Magma parts of this proof, and
of other proofs in this section, may be found at the URL 
listed in the Introduction.)

\begin{proposition}
\label{prop:gen3covs}
Let $\Xm_f/K\colon y^2 = f(x)$ be a genus-$3$ hyperelliptic
curve such that the form $f$ cancels the discriminants of 
all its quartic covariants.  Then $\Aut \Xm_f$ contains 
either $\DG_{12}$, $\CG_2 \times \DG_8$, or $\CG_{14}$.
\end{proposition}

\begin{proof}
Construct $C(f)\pm \kappa \cdot I(f) \cdot C'(f)$
such that $\deg(C) = \deg(I) + \deg(C')$, where $C$ and $C'$ 
run through the $14$ fundamental quartic covariants given
in~\cite[Table~1]{\LReleven}, where $I(f)$ equals either $1$
or a Shioda invariant $J_i(f)$, and where $\kappa$ runs through
the integers between $0$ and $10$.  We rewrite the discriminants
of these covariants in terms of Shioda invariants and add to them
the five Shioda relations~\cite[Theorem~3, p.~1042]{\shioda}.
Using Magma, we have been able to compute a Gr{\"o}bner basis
of this polynomial system, over $\QQ$, for the graded reverse 
lexicographical (or `grevlex') order $J_2 < J_3 < \ldots < J_{10}$
with weights $2$, $3$, \ldots, $10$. Upon removing
multiplicities, we obtain a basis with $22$ polynomials, 
of total degree between $8$ and $20$. One then checks, using
the stratum formulas from~\cite{\LReleven}, that the irreducible
components of the corresponding subscheme of the moduli space
either correspond to families of forms with discriminant zero 
or to strata of curves $\Xm_f$ such that $\Aut \Xm_f$ contain 
$\DG_{12}$, $\CG_2 \times \DG_8$, or $\CG_{14}$.
\end{proof}

We see from this that curves with automorphism group $\DG_{12}$,
$\CG_2 \times \DG_8$, or $\CG_{14}$ cannot have separable quartic 
covariants. In these cases, using Proposition~\ref{prop:ftoc}
and the normal models from Figure~\ref{tab:auto}, one can show:
\begin{itemize}
\item If $\Aut X$ is equal to $\DG_{12}$ or $\UG_6$ 
      then the sextic covariant $C_{3,6}=( ( f, f)_4, f)_5$ 
      has nonzero discriminant;
\item If $\Aut X$ contains $\CG_2 \times \DG_8$ or is
      equal to $\CG_{14}$ then there is no order-$4$ or
      order-$6$ covariant with three distinct roots.
\end{itemize}

The number of covariants considered in the proof of
Proposition~\ref{prop:gen3covs} --- namely, $1253$ --- is
not minimal, but the redundancy helped Magma during the
Gr{\"o}bner basis computations.  Nevertheless, similar 
computations show that we can easily reduce this number
for curves with automorphism group larger than $\CG_2$ 
(and moreover impose conditions on the automorphism groups
of the covariants; see Sections \ref{sec:covariant-approach} 
and \ref{sec:d4-case}).  For example, consider the following 
five quartic covariants:
\begin{align*}
C_{2,4}  &=       ( f, f)_6                     & C_{4,4}  &=   ( ( ( f, f)_4, f)_6, f)_4  \\
C_{3,4}  &=     ( ( f, f)_4, f)_6               & C_{4,4}' &=   ( ( ( f, f)_4, f)_4, f)_6  \\
         &                                      & C_{5,4}  &= ( ( ( ( f, f)_4, f)_6, f)_1, f)_7.  
\end{align*}
If $\Xm_f/K$ is a genus-$3$ hyperelliptic curve, we find that:
\begin{itemize}
\item If $\Aut \Xm_f \cong \DG_4$, one of the five covariants above
      has nonzero discriminant.
\item If $\Aut \Xm_f \cong \CG_4$, one of 
      $C_{2,4}$, $C_{3,4}$, $C_{4,4}$, and $C_{4,4}'$
      has nonzero discriminant.
\item If $\Aut \Xm_f \cong \CG_2^3$, one of 
      $C_{2,4}$, $C_{3,4}$, and $C_{4,4}$
      has nonzero discriminant.
\item If $\Aut \Xm_f \cong \CG_2 \times \CG_4$, 
      the covariant $C_{3,4}$
      has nonzero discriminant.
\end{itemize}

\begin{remark}
Similar conclusions hold for genus~$2$. Specifically, 
there is no quartic covariant with nonzero discriminant 
for the curves $\Xm_f/K$ such that 
$\DG_{12} \subset \Aut \Xm_f$ or $\Aut \Xm_f \simeq \CG_{10}$. 
Moreover, when $\Aut \Xm_f \simeq \DG_8$ then $(f,f)_4$ has 
nonzero discriminant, and when $\Aut \Xm_f \simeq \DG_4$ 
then at least one of $(f,f)_4$, $(((f,f)_2,f)_4,f)_4$, and
$((((f,f)_2,f)_3,f)_2,f)_6$ has nonzero discriminant.
\end{remark}

\section{Explicit descent for hyperelliptic curves}
\label{sec:desc-algebr-curv}

\subsection{Field of moduli and fields of definition}

Let $\Xm$ be a curve defined over $K$ of genus $g \geq 1$,
let $k$ be a subfield of $K$, and let $F$ be the prime
field of $K$.

\begin{definition}
\label{def:moduli}
The \emph{field of moduli} of $\Xm$, denoted $\MG_{\Xm}$,
is the subfield of $K$ fixed by 
$\{ \sigma \in \Aut K \mid  \Xm \simeq \Xm^{\sigma} \}.$
\end{definition}

We now restrict to hyperelliptic curves and we assume
that $p \ne 2$. Let $\Xm=\Xm_f$ be a hyperelliptic curve 
over $K$ given by a hyperelliptic polynomial $f$ of even degree~$n$. 
Our first task is to show that we can get
information on $\MG_{\Xm}$ through the invariants.

\begin{lemma}
Let $I_1,I_2$ be two invariants of the same degree 
for binary forms of degree~$n$. Assume that $I_1,I_2$
are defined over $F$ and that $I_2(f) \ne 0$. 
Then $\iota=I_1(f)/I_2(f)$ is an element of $\MG_{\Xm_f}$.
\end{lemma}

\begin{proof}
It is enough to prove that $\iota^{\sigma}=\iota$
for all $\sigma \in \Gal(K/\MG_{\Xm})$. By the definition 
of $\MG_{\Xm}$, there exists an isomorphism between
$\Xm$ and $ \Xm^{\sigma}$.  We have seen that such an
isomorphism induces an element $M \in \Isom(f,f^{\sigma})$. 
Therefore
\begin{equation*}
\iota^{\sigma}
   =\frac{I_1(f^{\sigma})}       {I_2 (f^{\sigma})}
   =\frac{I_1(\lambda \cdot M.f)}{I_2 (\lambda \cdot M.f)}
   =\iota . \qedhere
\end{equation*}
\end{proof}

It is not always practical to work with a fixed quotient of
invariants as above, since $I_2(f)$ may be zero. As shown
in~\cite{\LReleven}, it is better to work inside a weighted 
projective space, for elements of which one can define a canonical
representative as follows. Let $(I_1 : \ldots : I_m)$ be an 
$m$-tuple of degree-$d_i$ invariants of degree-$n$ binary forms,
where $m\ge 2$, and suppose each $I_i$ is defined over~$F$.
Let $f$ be a binary form of degree~$n$.  Let $d$ be the $\gcd$
of the degrees $d_i$ of the invariants $I_i$ whose values at $f$
are nonzero. Then there exist $c_i \in \ZZ$,
with $c_i=0$ if $I_i(f)=0$,
such that $\sum c_i d_i=d$. We then define $I=\prod_i I_i^{c_i}$.
The \emph{canonical representative} of $(I_1(f) : \ldots : I_m(f))$ is
$$(\frakI_1(f), \ldots, \frakI_m(f))
   =\left(\frac{I_1(f)}{I(f)^{d_1/d}}, \ldots, \frac{I_m(f)}{I(f)^{d_m/d}}\right)
     \in \MG_{\Xm}^m.$$
     
\begin{proposition} 
\label{prop:modvsinvariant}
Let $(I_1 : \ldots : I_m)$ be a set of generators
for $\ii_n$ defined over~$F$. Then
$$\MG_{\Xm}=F(\frakI_1(f),\ldots,\frakI_m(f)).$$
\end{proposition}

\begin{proof}
Let $\sigma \in \Gal(K/F(\frakI_1(f),\ldots,\frakI_m(f)))$. 
Since
$$(\frakI_1(f^{\sigma}), \ldots, \frakI_m(f^{\sigma}))
   = (\frakI_1(f), \ldots, \frakI_m(f)),$$
and since $\ii_n$ separates the orbits of
separable forms \cite[p.~78]{\mumfordfogarty},
there exists a matrix $M \in \GL_2(K)$ such that 
$M.f \sim f^{\sigma}$, 
hence an isomorphism between $\Xm_f$ and $\Xm_f^{\sigma }$.
\end{proof}

With our current knowledge of invariants, we are then able to
compute $\MG_{\Xm_f}$ for $n=6,8,10$. However, in the following
applications to descent we will see that we often do not need a
complete set of invariants.

\begin{definition}
We say that $k$ is a \emph{field of definition} of $\Xm$ 
if there exists a curve $\xx/k$ such that $\xx$ is $K$-isomorphic 
to~$\Xm$. The curve $\xx/k$ is a model of $\Xm$ over $k$ and we 
call a geometric isomorphism between the
two curves a \emph{descent isomorphism}.
\end{definition}

A classical problem is to determine the smallest field of definition
of a curve. Assuming for simplicity that every subfield of $K$
is perfect, if $\MG_{\Xm}$ is a field of definition then it is the
smallest possible field of definition, because 
it is the intersection of all the fields of definition
(see \cite{\koizumi} or \cite[Theorem~1.5.8]{\hugginsphd}).
There might be an obstruction for $\MG_{\Xm}$ being a field of 
definition, but if there is none we will denote by $\xx$ a model of
$\Xm$ over $\MG_{\Xm}$. In the case of hyperelliptic curves of odd 
genus, there is a subtlety:  The curve $\xx$ does not necessarily admit
a hyperelliptic equation.  However, if it does, we will say that $\Xm$
can be \emph{hyperelliptically defined over $\MG_{\Xm}$}, and we 
denote by $\frakf \in \MG_{\Xm}[x]$ a hyperelliptic polynomial 
associated to this model.

One can find in the literature several sufficient conditions
for a curve to be hyperelliptically defined over $\MG_{\Xm}$.
For instance, it is always the case when $K$ is the algebraic
closure of a finite field  (see~\cite[Corollary~2.11]{\huggins}). 
Over an arbitrary algebraically closed field~$K$, the work
of Huggins~\cite{\huggins} shows that if the reduced 
automorphism group is noncyclic then the curve can be 
hyperelliptically defined over its field of moduli. 
For $g=2$, it has been proved that if the reduced 
automorphism group is nontrivial, then the curve can be
hyperelliptically defined over its field of moduli~\cite{\CAQU}.
This is also the case for $g=3$, except for curves with 
automorphism group isomorphic to $\DG_4$ (see~\cite{\LReleven}
and Section~\ref{sec:d4-case}).

\subsection{Explicit hyperelliptic descent}
Now let $\Xm_f$ be a hyperelliptic curve over $K$ that can be
hyperelliptically defined over $\MG_{\Xm}$. We want to find
$\frakf \in \MG_{\Xm}[x]$ and $A \in \GL_2(K)$ such that $\frakf \sim A.f$.
The first task is of course to compute $\MG_{\Xm}$. As we have seen, this 
can be done if we have a set of generators for the invariants of the form~$f$.
However, if we do not have a full set of generators, and instead 
have only some invariants $(I_1,\ldots,I_m)$ over $F$ with $m \geq 2$,
we can always try to hyperelliptically descend $\Xm_f$ over the 
field $k$ generated by $(\frakI_1(f),\ldots,\frakI_m(f))$.
Since $k \subset \MG_{\Xm}$, if this can be achieved, we are done.

\subsubsection{The cocycle approach}
\label{sec:classical-approach-2}

The direct approach relies on the following slightly modified 
version of Weil's cocycle relations (see~\cite{\LReleven}).

\begin{lemma}
The curve $\Xm_f$ can be hyperelliptically defined over $k$
if and only if there exists a finite extension $k'/k$ such
that for all $\sigma \in \Gal(K/k)$, there exists
$M_{\sigma} \in \GL_2(k')$ 
such that $M_\sigma \in \Isom_{k'}(f,f^{\sigma})$ 
and such that for all $\sigma, \tau \in \Gal(K/k)$,
we have $M_{\sigma \tau}= M_\sigma^\tau M_\tau.$
\end{lemma}

Assume that $\Xm_f$ can be hyperelliptically defined
over $k$ and let $\phi\colon \Xm_f \to \Xm_{\frakf}$
be a descent isomorphism.  It induces a matrix 
$\tilde{A} \in \Isom_K(f,\frakf) \subset \PGL_2(K)$.
If we choose a representative $A \in GL_2(K)$ of $\tilde{A}$,
we can define $M_\sigma = (A^{-1})^{\sigma} A$ for all
$\sigma \in \Gal(K/k)$.  It is easy to check that this
choice of $M_\sigma$ satisfies all the hypotheses of the
lemma. Moreover, if $A$ is defined over a Galois extension
$L/k$ then $k' \subset L$, and we have $M_{\sigma}=\Id$
for all $\sigma \in \Gal(K/k)$ such that $\sigma_{|L}=\Id$. 
Conversely, the crucial step to construct such an $A$ is 
to identify a Galois extension $L/k$ satisfying this 
property, since in this case one can use an explicit
version of Hilbert 90 as in~\cite[Proposition~3, p.~159]{\serre}:
For a general matrix $P \in \GL_2(k')$ the matrix
\begin{equation}
\label{hilbert90}
A=\sum_{\tau \in \Gal(L/k)} P^{\tau}  M_{\tau}
\end{equation}
gives a descent morphism.

\begin{lemma}
Assume that $f$ is defined over an extension $k'$ of $k$. If
$\Aut_K f =\{\Id\}$ then we can take $L$ to be the Galois closure of~$k'/k$.
\end{lemma}

\begin{proof}
We have to prove that $A$ can be defined over such an~$L$.
Let $A'$ be induced by a descent morphism. Since 
$A' \in \Isom_{K}(f,\frakf)$, we have 
$((A')^{-1})^{\sigma} A' \in \Isom_K(f,f^{\sigma})=\Aut_K f$ 
for all $\sigma \in \Gal(K/L)$;
hence there exists $\lambda_{\sigma} \in K^*$ such that 
$(A')^{\sigma}= \lambda_{\sigma} \cdot A'$. One can easily
check that the $\lambda_{\sigma}$ satisfy a cocycle relation,
so there exists $e \in K^*$ such that $\lambda_{\sigma}=e/e^{\sigma}$
for all~$\sigma$. We then define $A=e \cdot A'$, and we are done.
\end{proof}

As far as we know, there is no easy way to determine such
an $L$ when the automorphism group is nontrivial 
(but see \cite{\LReleven} for the case when $k$ is a finite field).
Na\"{\i}vely, one would expect to be able to construct the 
cocycle over the field $L_0$ over which all isomorphisms between
$f$ and its conjugates are defined. Typically, what then happens
is the following:
Let $\sigma \in \Gal (L_0 / k)$ be an element of order~$n$. 
Then usually no $M_{\sigma}$ exists over $L_0$ such that
the cocycle condition 
$1 = M_{\sigma^n} = M^{\sigma^{n-1}} \cdots M^{\sigma} \cdot M$
is satisfied. We have to work with matrices of the form 
$\lambda M_{\sigma}$, where $\lambda$ belongs to a quadratic 
extension $L$ of $L_0$. This enlarges the field and the 
Galois group, which may in turn give rise to more problems
of the same type.  Even if this problem can be resolved, 
the computation of Equation~\eqref{hilbert90} is time-consuming
and limited to extensions of small degree (less than $50$)
in practice.  In the next section, we present a new idea that
works extremely well to get around these difficulties in 
certain cases.

\begin{remark}
In the odd genus case, it turns out that if we only want 
$\Xm_f$ to have a model over~$k$, instead of a
hyperelliptic model, then the cocycle condition is replaced
by the condition $M_{\sigma \tau} \sim M_{\sigma}^{\tau} M_{\tau}$.
However, even in this case we do not know a general method to address the 
problem effectively.
\end{remark}

\subsubsection{The covariant approach}
\label{sec:covariant-approach}

Using covariants, we can sometimes reduce the problem of 
descent for $\Xm_{f}$ to a descent problem for a curve 
of lower genus.

\begin{theorem}
\label{th:covdescent}
Assume that there exists a covariant $C$ of order 
$r \geq 4$ such that $c=C(f)$ is a hyperelliptic polynomial,
and let $\Xm_c\colon y^2=c(x)$ be the associated curve. 
Then $\MG_{\Xm_c} \subset \MG_{\Xm_f}$.

Moreover, if $\Xm_c$ is hyperelliptically defined over $\MG_{\Xm_c}$, then
$\Xm_f$ is hyperelliptically defined over an extension of $\MG_{\Xm_f}$
of degree at most $[ \Aut_K c : \Aut_K f ]$. 

In particular, if $\Aut_K c =\Aut_K f$ and if $\Xm_c$ is hyperelliptically
defined over~$\MG_{\Xm_c}$, then $\Xm_f$ is hyperelliptically defined over
$\MG_{\Xm_f}$.
\end{theorem}

\begin{proof}
Let $\sigma$ be an element of the group $\Gamma=\Gal(K/\MG_{\Xm_f})$.
Then there exists a $K$-isomorphism between $\Xm_f$ and 
$\Xm_f^{\sigma}$ which induces a matrix $M \in \Isom_K(f,f^{\sigma})$.
Since we have the inclusion
$\Isom_K(f,f^{\sigma}) \subset \Isom_K(c,c^{\sigma})$ by 
Proposition~\ref{prop:ftoc}, we get a $K$-isomorphism between
$\Xm_c$ and $\Xm_c^{\sigma}$, so $\MG_{\Xm_c} \subset \MG_{\Xm_f}$.
  
Assume now that $\Xm_c$ can be hyperelliptically defined
over $\MG_{\Xm_c}$ as $\Xm_{\frakc}$ for some form
$\frakc \in \MG_{\Xm_c}[x]$. There exists $A \in \Isom_K(c,\frakc)$.
Let us consider $h=A.f$, which we can assume to be
monic. We want to prove that $h$ is defined over an extension 
of $\MG_{\Xm_f}=\MG_{\Xm_{h}}$ of degree at most 
$$\ell=\#(\Aut_K c/\Aut_K f) = \#(\Aut_K \frakc/\Aut_K h).$$
First note that $C(h) \sim A.C(f) \sim \frakc$.
Let $H \subset \Gamma$ be the subgroup consisting of the
automorphisms $\sigma$ such that $h \sim h^{\sigma}$. 
Since we have assumed that $h$ is monic, we even have $h = h^{\sigma}$.
We must show that $\# \Gamma/H \leq \ell$. 
To this end, we note that $\frakc^{\sigma}=\frakc$
for all $\sigma \in \Gamma$.  Hence we can associate to 
each $\sigma \in \Gamma$ a matrix 
$M \in \Isom_K(h,{h}^{\sigma}) \subset \Aut_K \frakc$.
In fact, this association gives rise to a well-defined 
class of $\Aut_K \frakc/\Aut_K h$, so we have defined 
a map $\rho$ from $\Gamma$ to $\Aut_K \frakc/\Aut_K h$.
If $\rho(\sigma)=\rho(\sigma')$ then we have $h^{\sigma} \sim h^{\sigma'}$,
and hence $\sigma^{-1} \sigma' \in H$. Therefore $\rho$ induces an injective map
from $\Gamma/H$ to $\Aut_K \frakc/\Aut_K h$, and we get our result.
\end{proof}

To use the theorem in a constructive way, we need a covariant
that has a finite automorphism group and for which we know
how to find a hyperelliptic model over its field of moduli.
We give some examples in Sections~\ref{sec:applications-genus-3}
and~\ref{sec:fuertes-examples}.

\begin{remark}
The fields of moduli of $\Xm_f$ and $\Xm_c$ may be
different, even when the automorphism groups of the forms 
are the same. For instance, let $r$ be a root of $t^2+2t+16/9=0$
and let $f$ be the form 
$$f = (x^4 + r x^2 z^2 + z^4) (x^4 - 3 r x^2 z^2 + z^4);$$
then the field of moduli of $f$ is $\QQ(r)$, while
the field of moduli of
$$c=(f,f)_6= (16/49) x^4 + (992/441) x^2 + (16/49)$$ is~$\QQ$.
Using the programs of~\cite{\LReleven}, one sees that
$\Aut_K f=\Aut_K c \simeq \DG_4$.
\end{remark}

\subsection{Application to genus-\texorpdfstring{$3$}{3} hyperelliptic curves}
\label{sec:applications-genus-3}

In~\cite{\LReleven}, the two first authors give algorithms
for reconstructing genus-$3$ hyperelliptic models from 
given invariants.  These models are defined over the
field of moduli, with the notable exception of the 
$2$-dimensional stratum $\CG_2^3$ and the
$3$-dimensional stratum~$\DG_4$. 
As an illustration of our strategy, we see how our
method applies in these remaining cases.

\subsubsection{Descent of curves with automorphism group \texorpdfstring{$\CG_2^3$}{C2 cubed}}
\label{sec:c23-case}

Let $\Xm/K\colon y^2=f(x)$ be a genus-$3$ hyperelliptic curve 
with automorphism group isomorphic to $\CG_2^3$.  Since the
reduced automorphism group is not cyclic, \cite{\huggins} 
shows that $\Xm$ can be hyperelliptically defined over
its field of moduli.  In \cite{\LReleven}, we showed how to 
construct a hyperelliptic equation for a model over an 
extension of the field of moduli of degree at most~$3$.
Using covariants, we can now give a method to get an
equation over the field of moduli itself.

In Section~\ref{sec:all-hyper-genus}, we checked that 
at least one of the quartic covariants in the list 
$\{ C_{2,4}(f),$ $C_{3,4}(f),$ $C_{4,4}(f) \}$ has nonzero 
discriminant.  Moreover, by Proposition~\ref{prop:autquartics},
we see that the automorphism group of such a quartic
is equal to $\DG_4$ if the quartic invariants $I$ and $J$
are both nonzero. Using some formal computations 
(see the Magma scripts available at the URL listed in
the Introduction), we checked that it is always the 
case that at least one of the three covariants has nonzero
discriminant and $I$ and $J$ nonzero.
Since $\Aut_K(f) \simeq \DG_4$ we can use the approach of 
Theorem~\ref{th:covdescent} to find a hyperelliptic equation 
$y^2=\frakf(x)$ over the field of moduli. The procedure
can actually be applied to a generic element of the 
family, but the result is too large to be written down here; 
instead, we present an example.

\begin{example}
\label{ex:c23g3}
When we evaluate the parametrization formulas given
in \cite{\LReleven} for the stratum $\CG_2^3$ at 
$t=0$ and $u=1$, we find the rational point
\begin{multline*}
(j_2: j_3:\ldots:j_{10})\\
     = \left(0 :0 :-\frac{   25}{     98}
                  :-\frac{   25}{     98}
                  :-\frac{  225}{   2744}
                  :-\frac{   25}{   1372}
                  :-\frac{  225}{ 134456}
                  : \frac{ 1125}{  76832}
                  : \frac{15125}{3764768} \right)
\end{multline*}
in the moduli space. This gives rise to the curve $\Xm\colon y^2 = f$ with
\begin{multline*}
   f =  (-32 \alpha^2 +420 \alpha - 2275) x^8 / 160
      + (-12 \alpha^2 +140 \alpha -  700) x^6 / 25\\
      + \alpha x^4 + x^2 + (16 \alpha^2 + 280 \alpha - 2275)/12250
\end{multline*}
over $\QQ(\alpha)$, where
$\alpha^3-(35/2)\alpha^2 + (1925/16)\alpha - (18375/64) = 0.$
By Proposition~\ref{prop:modvsinvariant}, we have $\MG_{\Xm}=\QQ$.
  
Let $c$ be the covariant $(f, f)_6$. We find
$$
      c =  \frac{-16 \alpha^2+180 \alpha-875}{280} x^4
         + \frac{24 \alpha^2-630 \alpha+3150}{1225} x^2 z^2
         + \frac{4 \alpha+35}{490} z^4,
$$
so that $I = -75/49$ and $J = -2025/343$.  It follows that 
$\frakc = x^3 z+{{25}/9}\ x z^3+{{25}/9}\ z^4$
is $\GL_2(\Qbar)$-equivalent to $c$, 
is defined over $\MG_{\Xm} = \QQ$, and 
satisfies $\Aut_{\Qbar} \frakc\simeq \DG_4$.
The direct approach of Section~\ref{sec:classical-approach}
explicitly finds a $\Qbar$-isomorphism $M$ between $c$
and~$\frakc$.  Its inverse $M^{-1}$ is equal to $(m_{i,j})_{i,j}$, where
\begin{align*}
  m_{11} &= 110250,\\
  m_{12} &= ( 3360 \alpha^2-58800 \alpha+147000 )  \beta^2
            -16800 \alpha^2+ 147000 \alpha-18375, \\
  m_{21} &= ( -2064 \alpha^2+24780 \alpha-60900 )  \beta^3
            + ( -3120 \alpha^2+ 67200 \alpha-375375 )  \beta,      \\
  m_{22} &=  ( -5840 \alpha^2+74900 \alpha-280000 )  \beta^3+ ( 16880 
          \alpha^2-173600 \alpha+487375 )  \beta .
\end{align*}
Here $\beta$ satisfies
$$ \beta^4+ \frac{ 32 \alpha^2-280 \alpha+350 }{175} \beta^2-
    \frac{176 \alpha^{ 2}-1820 \alpha+7350}{175} = 0.$$
We compute the monic form $\frakf \sim M.f$:
\begin{multline*}
\frakf =  x^8+160x^7-560x^6-2800x^5+64750x^4-91000x^3\\
            +3010000x^2-2225000x-9696875.
\end{multline*}
So $y^2=\frakf(x)$ is a model of $\Xm$ over $\MG_{\Xm}=\QQ$.
\end{example}

\subsubsection{Descent of curves with automorphism group $\DG_4$}
\label{sec:d4-case}
It is proved in \cite[Chapter~5]{\hugginsphd} that
there may be an obstruction for a genus-$3$ hyperelliptic
curve over $K$ with automorphism group isomorphic to $\DG_4$ 
to have a model over its field of moduli. In \cite{\LReleven},
we were able to construct a model of such curves over an
extension of the field of moduli of degree at most~$8$. 
Using Theorem~\ref{th:covdescent}, we find:

\begin{proposition}
Let $\Xm_f$ be a genus $3$ hyperelliptic curve over $K$
with automorphism group isomorphic to $\DG_4$. Then there
exists an explicit model of $\Xm$ over an at most
quadratic extension of $\MG_{\Xm}$.
\end{proposition}

\begin{proof}
Applying the methods of Proposition~\ref{prop:gen3covs} 
to the stratum $\DG_4$ shows that at least one of the
five binary covariants $C_{2,4}(f)$, $C_{3,4}(f)$, $C_{4,4}(f)$, 
$C'_{4,4}(f)$, $C_{5,4}(f)$ has not only a discriminant
different from $0$, but also $I(f) \neq 0$ and $J(f) \neq 0$.  
(The computations can be found in the the Magma scripts available
at the URL given in the Introduction.)
One then combines Proposition~\ref{prop:autquartics} and
Theorem~\ref{th:covdescent}.
\end{proof}
We plan to investigate how to apply the theory of twists to the binary
quartics used in the application of Theorem~\ref{th:covdescent} to give a
precise characterization of the obstruction to the descent on the field of
moduli.

\subsection{Application to a family of Fuertes-Gonz\'alez-Diez
            in genus \texorpdfstring{$5$}{5}}
\label{sec:fuertes-examples}

Let $k$ be the degree-$3$ Galois extension of $\QQ$ defined by the
irreducible polynomial $t^3 - 3t + 1$. Let $r_1, r_2, r_3$ be the
roots of this polynomial in~$k$. Then, as in~\cite{\fuertesgonzalez},
we can consider the family
\begin{equation}
\label{fgdfam}
y^2 = \prod_{i = 4}^6 
   \Bigl(x^4-2  \Bigl(1-2\,  \frac{r_3-r_1}{r_3-r_2} \, \frac{q_i-r_2}{q_4-r_1}\Bigr)  x^2+1\Bigr)
\end{equation}
of genus-$5$ hyperelliptic curves, with $q_4, q_5, q_6$ in~$\QQ$.
It was proved in \cite{\fuertesgonzalez} that the members of this 
family have field of moduli equal to $\QQ$ and automorphism group
isomorphic to $\CG_2^3$.  Moreover, it was claimed in 
\cite{\fuertesgonzalez} that these curves cannot be hyperelliptically
defined over $\QQ$, in contradiction with~\cite{\huggins}.  However,
the proof turns out to contain a subtle error.  Still, the explicit 
descent of any of the member of the family was extremely hard.

As in Example~\ref{ex:c23g3}, we can use Theorem \ref{th:covdescent}
to construct an explicit descent for the curves in this family. 
For this particular family, the descent can even be performed
uniformly to yield a general expression in $q_4, q_5, q_6$.
Let $F = k (q_4,q_5,q_6)$ be the rational function field 
over $k$ in three indeterminates, and define the
binary quartic form $f \in F[x,z]$ as the homogenization
of the right hand side of Equation~\eqref{fgdfam}. 
Let $c$ be the transvectant $(f,f)_{10}$.  Then $c$ is a 
covariant of order $4$ with nonzero discriminant and 
nonzero $I(c)$ and $J(c)$, and hence has automorphism group~$\DG_4$.
The field of moduli of $X_c$ is contained in the field of moduli
of $X_f$, which is a subfield of $\QQ(q_4,q_5,q_6)$;
therefore the quartic $\mathfrak{c}$ as in Equation~\eqref{quarticsplit}
is defined over $\QQ (q_4 , q_5, q_6)$ and is
$\GL_2(\Fbar)$-equivalent to $c$. 

Now let $L$ be the degree-$4$ extension of $F$
defined by the dehomogenization of~$c$.  From 
Proposition~\ref{prop:iso-quartic}, we can explicitly 
construct an $L$-isomorphism between $c$ and~$\frakc$.
This transformation gives a descent of the curve 
corresponding to~$c$, which by Theorem~\ref{th:covdescent}
also yields a descent of the curve corresponding to~$f$. 
The resulting expression, though indeed defined over the 
rationals, is huge and impossible to give here. 
(The computations above, their final result, and the 
program to compute the descent of any given
specialization are available at the URL listed 
in the Introduction.)  However, we can give an 
example for a specialization.

\begin{example}
Take $q_4 = 1$, $q_5 = 2$, $q_6 = 3$. The hyperelliptic equation over $\QQ$ is
\begin{multline*}
    y^2 =   199950247575 x^{12} - 296949924611352 x^{11} - 66659816245812750 x^{10} \\
        - 15421975495507360656 x^9 + 2005635519424553708745 x^8 \\
        + 130792088864772419461200 x^7 + 44148454149188354317253820 x^6 \\
        - 9718847083908693649803959136 x^5 + 93749472927036312839424054441 x^4 \\
        + 86331359417888600607650948443656 x^3 \\
        - 7423912080663182513045938205161326 x^2 \\
     + 249511197641168404939510946041515184 x \\
     - 3006656143858472317763973580984260681 .
\end{multline*}
\end{example}

\section*{Acknowledgments}
The authors acknowledge support by grant ANR-09-BLAN-0020-01.

\nocite{Shaska2005}
\nocite{MEGA90}
\nocite{StonyBrook1969}
\nocite{ANTS2004}
\nocite{FOCS2008}

\newcommand{\etalchar}[1]{$^{#1}$}
\newcommand{\SortNoop}[1]{}


\begin{thebibliography}{{\SortNoop{Rijnswou}}R01}

\bibitem[ABF{\etalchar{+}}71]{StonyBrook1969}
Lars~V. Ahlfors, Lipman Bers, Hershel~M. Farkas, Robert~C. Gunning, Irwin Kra,
  and Rauch~Harry E., editors.
\newblock {\em Advances in the theory of {R}iemann surfaces}, volume~66 of {\em
  Annals of Mathematics Studies}, Princeton, N.J., 1971. Princeton University
  Press.
\newblock Proceedings of the 1969 Stony Brook Conference.

\bibitem[BCP97]{MAGMA}
Wieb Bosma, John Cannon, and Catherine Playoust.
\newblock The {M}agma algebra system. {I}. {T}he user language.
\newblock {\em J. Symbolic Comput.}, 24(3-4):235--265, 1997.
\newblock Computational algebra and number theory (London, 1993). Software
  available at \url{http://magma.maths.usyd.edu.au/}.

\bibitem[Bed07]{Bedratyuk2007}
Leonid Bedratyuk.
\newblock On complete system of invariants for the binary form of degree 7.
\newblock {\em J. Symbolic Comput.}, 42(10):935--947, 2007.

\bibitem[BP10a]{BrouwerPopoviciu2010b}
Andries~E. Brouwer and Mihaela Popoviciu.
\newblock The invariants of the binary decimic.
\newblock {\em J. Symbolic Comput.}, 45(8):837--843, 2010.

\bibitem[BP10b]{BrouwerPopoviciu2010a}
Andries~E. Brouwer and Mihaela Popoviciu.
\newblock The invariants of the binary nonic.
\newblock {\em J. Symbolic Comput.}, 45(6):709--720, 2010.

\bibitem[Bue04]{ANTS2004}
Duncan Buell, editor.
\newblock {\em Algorithmic number theory}, volume 3076 of {\em Lecture Notes in
  Computer Science}, Berlin, 2004. Springer-Verlag.

\bibitem[CF09]{CremonaFisher2009}
J.~E. Cremona and T.~A. Fisher.
\newblock On the equivalence of binary quartics.
\newblock {\em J. Symbolic Comput.}, 44(6):673--682, 2009.

\bibitem[CQ05]{CardonaQuer2005}
Gabriel Cardona and Jordi Quer.
\newblock Field of moduli and field of definition for curves of genus 2.
\newblock In Shaska \cite{Shaska2005}, pages 71--83.

\bibitem[Cr{\"o}02]{Croni2002}
H.~Cr{\"o}ni.
\newblock {\em Zur {B}erechnung von {K}ovarianten von {Q}uantiken}.
\newblock PhD thesis, Universit\"{a}t des Saarlandes, 2002.

\bibitem[Dix90]{Dixmier1990}
Jacques Dixmier.
\newblock Quelques aspects de la th\'eorie des invariants.
\newblock {\em Gaz. Math.}, (43):39--64, 1990.
\newblock Translated by J.-R. Billuard.

\bibitem[DL86]{DixmierLazard1985/1986}
J.~Dixmier and D.~Lazard.
\newblock Le nombre minimum d'invariants fondamentaux pour les formes binaires
  de degr\'e {$7$}.
\newblock {\em Portugal. Math.}, 43(3):377--392, 1985/86.

\bibitem[Ear71]{Earle1971}
Clifford~J. Earle.
\newblock On the moduli of closed {R}iemann surfaces with symmetries.
\newblock In Ahlfors et~al. \cite{StonyBrook1969}, pages 119--130.

\bibitem[FGD06]{FuertesGonzalez2006}
Y.~Fuertes and G.~Gonz{\'a}lez-Diez.
\newblock Fields of moduli and definition of hyperelliptic covers.
\newblock {\em Arch. Math. (Basel)}, 86(5):398--408, 2006.

\bibitem[{\SortNoop{Gall}}G88]{Gall1888}
August~Freiherr {\SortNoop{Gall}}von~Gall.
\newblock Das vollst\"andige {F}ormensystem der bin\"aren {F}orm 7 {$^{ter}$}
  {O}rdnung.
\newblock {\em Math. Ann.}, 31(3):318--336, 1888.

\bibitem[Gor68]{Gordan1868}
Paul Gordan.
\newblock Beweis, dass jede {C}ovariante und {I}nvatiante einer bin{\"a}ren
  {F}orm eine ganze {F}unction mit numerischen {C}oefficienten einer endlichen
  {A}nzahl solcher {F}ormen ist.
\newblock {\em J. Reine Angew. Math.}, 69:323--354, 1868.

\bibitem[GY03]{GraceYoung1903}
J.~H. Grace and A.~Young.
\newblock {\em The algebra of invariants}.
\newblock Cambridge University Press, Cambridge, 1903.

\bibitem[Hes04]{Hess2004}
F.~Hess.
\newblock An algorithm for computing isomorphisms of algebraic function fields.
\newblock In {\em Algorithmic number theory}, volume 3076 of {\em Lecture Notes
  in Comput. Sci.}, pages 263--271. Springer, Berlin, 2004.

\bibitem[Hug05]{Huggins2005}
Bonnie~Sakura Huggins.
\newblock {\em Fields of Moduli and Fields of Definition of Curves}.
\newblock PhD thesis, University of California, Berkeley, 2005.

\bibitem[Hug07]{Huggins2007}
Bonnie Huggins.
\newblock Fields of moduli of hyperelliptic curves.
\newblock {\em Math. Res. Lett.}, 14(2):249--262, 2007.

\bibitem[Igu60]{Igusa1960}
Jun-ichi Igusa.
\newblock Arithmetic variety of moduli for genus two.
\newblock {\em Ann. of Math. (2)}, 72:612--649, 1960.

\bibitem[Ins08]{FOCS2008}
Institute of Electrical and Electronics Engineers.
\newblock {\em 49th {IEEE} {S}ymposium on {F}oundations of {C}omputer
  {S}cience---{FOCS} 2008}, Los Alamitos, CA, 2008. IEEE Computer Society.

\bibitem[Koi72]{Koizumi1972}
Shoji Koizumi.
\newblock The fields of moduli for polarized abelian varieties and for curves.
\newblock {\em Nagoya Math. J.}, 48:37--55, 1972.

\bibitem[KU08]{KedlayaUmans2008}
Kiran~S. Kedlaya and Christopher Umans.
\newblock Fast modular composition in any characteristic.
\newblock In {\em Proceedings of the 49th {IEEE} {S}ymposium on {F}oundations
  of {C}omputer {S}cience held in {P}hiladelphia, {O}ctober 25--28, 2008\/}
  \cite{FOCS2008}, pages 146--155.

\bibitem[LR]{LercierRitzenthaler2008}
Reynald Lercier and Christophe Ritzenthaler.
\newblock {\em Invariants and reconstructions for genus $2$ curves in any
  characteristic}.
\newblock Available in MAGMA 2.15~\cite{MAGMA} and later.

\bibitem[LR12]{LercierRitzenthaler2012}
Reynald Lercier and Christophe Ritzenthaler.
\newblock Hyperelliptic curves and their invariants: geometric, arithmetic and
  algorithmic aspects.
\newblock {\em J. Algebra}, 372:595--636, 2012.

\bibitem[Mes91]{Mestre1991}
Jean-Fran{\c{c}}ois Mestre.
\newblock Construction de courbes de genre {$2$} \`a partir de leurs modules.
\newblock In Mora and Traverso \cite{MEGA90}, pages 313--334.

\bibitem[MF82]{MumfordFogarty1982}
David Mumford and John Fogarty.
\newblock {\em Geometric invariant theory}, volume~34 of {\em Ergebnisse der
  Mathematik und ihrer Grenzgebiete [Results in Mathematics and Related
  Areas]}.
\newblock Springer-Verlag, Berlin, second edition, 1982.

\bibitem[MT91]{MEGA90}
Teo Mora and Carlo Traverso, editors.
\newblock {\em Effective methods in algebraic geometry}, volume~94 of {\em
  Progress in Mathematics}.
\newblock Birkh\"auser Boston Inc., Boston, MA, 1991.
\newblock Papers from the symposium (MEGA-90) held in Castiglioncello, April
  17--21, 1990.

\bibitem[Olv99]{Olver1999}
Peter~J. Olver.
\newblock {\em Classical invariant theory}, volume~44 of {\em London
  Mathematical Society Student Texts}.
\newblock Cambridge University Press, Cambridge, 1999.

\bibitem[{\SortNoop{Rijnswou}}R01]{Rijnswou2001}
Sander~Matthijs {\SortNoop{Rijnswou}}van~Rijnswou.
\newblock {\em Testing the Equivalence of Planar Curves}.
\newblock PhD thesis, Technische Universiteit Eindhoven, April 2001.

\bibitem[Ser68]{Serre1968}
Jean-Pierre Serre.
\newblock {\em Corps locaux}.
\newblock Hermann, Paris, 1968.
\newblock Deuxi{\`e}me {\'e}dition, Publications de l'Universit{\'e} de
  Nancago, No. VIII.

\bibitem[Sha05]{Shaska2005}
Tanush Shaska, editor.
\newblock {\em Computational aspects of algebraic curves}, volume~13 of {\em
  Lecture Notes Series on Computing}.
\newblock World Scientific Publishing Co. Pte. Ltd., Hackensack, NJ, 2005.
\newblock Papers from the conference held at the University of Idaho, Moscow,
  ID, May 26--28, 2005.

\bibitem[Shi67]{Shioda1967}
Tetsuji Shioda.
\newblock On the graded ring of invariants of binary octavics.
\newblock {\em Amer. J. Math.}, 89:1022--1046, 1967.

\bibitem[Shi72]{Shimura1972}
Goro Shimura.
\newblock On the field of rationality for an abelian variety.
\newblock {\em Nagoya Math. J.}, 45:167--178, 1972.

\bibitem[Wei56]{Weil1956}
Andr{\'e} Weil.
\newblock The field of definition of a variety.
\newblock {\em Amer. J. Math.}, 78:509--524, 1956.

\end{thebibliography}
\end{document}